\documentclass[a4paper,twoside,english,10pt]{amsart}
\usepackage{lmodern} 
\usepackage[margin=1.3in]{geometry} 
\usepackage{graphicx}              
\usepackage{amsmath}              
\usepackage{amsfonts}              
\usepackage{amsthm}                
\usepackage{verbatim}              
\usepackage{indentfirst}           
\usepackage{underscore}
\usepackage{enumitem}
\usepackage{mathtools} 
\usepackage{amssymb}
\usepackage{hyperref}
\usepackage{xcolor}
\usepackage[all]{xy}
\usepackage[utf8]{inputenc}
\newtheorem{thm}{Theorem}
\newtheorem{lem}[thm]{Lemma}

\newtheorem{cor}[thm]{Corollary}

\theoremstyle{remark}
       \newtheorem*{rmk}{Remark}
\theoremstyle{remark}

 \newcommand{\ord}{\operatorname{ord}}  
    
    \newcommand{\aff}{\operatorname{Aff}}  
  
\usepackage{amssymb,fge}

\title{Free semigroups of power series}
\author{Wade Hindes}
\begin{document}
\begin{abstract} Given power series $f_1,\dots,f_r\in z^2K[[z]]$ over a field $K$, we show that the semigroup $\langle f_1,\dots,f_r\rangle$ generated by the $f$'s under composition is free of rank $r$ whenever the multiplicative semigroup $\langle\textup{lc}(f_1),\dots,\textup{lc}(f_r)\rangle$ in $K^\times$ is free commutative of rank $r$; here $\textup{lc}(f)$ denotes the first nonzero coefficient of $f$. To do this, we associate an affine linear transformation to every $F\in\langle f_1,\dots,f_r\rangle$ and then apply a ping-pong lemma. In particular, combining this with \cite{PappalardiShaShparlinskiStewart2018} implies that most semigroups of polynomials over a number field $K$ are free, and we make this statement precise through a counting argument.     
\end{abstract}
\maketitle

Let $K$ be a field and let $K[[z]]$ be the ring of formal power series in one variable over $K$. Then, given a nonzero $f=\sum_{n\geq0} a_nz^n\in K[[z]]$, we let $d=\min\{n: a_n\neq0\}$ and define $\textup{lc}(f):=a_d$ and $\ord(f):=d$ to be the leading coefficient and order vanishing at $z=0$ of $f$, respectively. Given $f_1,\dots,f_r\in z\,K[[z]]$ we let $\langle f_1,\dots,f_r\rangle$ denote the semigroup of power series generated by the $f$'s under composition. Finally, for $\mathbf{a}=(a_1,\dots,a_r)\in(K^\times)^r$ and $\mathbf{n}\in\mathbb{Z}^r$, write $\mathbf{a}^{\mathbf{n}}:=a_1^{n_1}\cdots a_r^{n_r}$ and say that $a_1,\dots,a_r$ are \emph{multiplicatively independent} if
$\mathbf{a}^{\mathbf{n}}=1$ implies $\mathbf{n}=\mathbf{0}$. Equivalently, the multiplicative semigroup $\langle a_1,\dots,a_r\rangle$ in $K^\times$ generated by the $a$'s is free commutative of rank $r$. Then we have the following result; compare to results in \cite{beaumont2025uniformtitsalternativeendomorphisms,MR4780496, Zieve} for semigroups of rational functions on two generators and to \cite{pakovich2026right} for other structural results on semigroups of power series.  
\begin{thm}\label{thm:free+power+series} 
Let $K$ be a field and let $f_1,\dots,f_r\in z^2K[[z]]$ be nonzero. If the leading coefficients $\textup{lc}(f_1),\dots,\textup{lc}(f_r)$ are multiplicatively independent elements of $K^\times$, then $\langle f_1,\dots,f_r\rangle$ is a free semigroup of rank $r$.  
\end{thm}
\begin{rmk}
After completion of this paper, Fedor Pakovich pointed out a second proof of Theorem \ref{thm:free+power+series}, based on the coefficient estimates in the proof of \cite[Theorem 2.3]{pakovich2022sharing}. With his permission, we include this argument at the end of the paper.
\end{rmk}
We break the proof into pieces in hopes that it may help formulate generalizations. In what follows, $\textup{lt}(f):=\textup{lc}(f)z^{\ord(f)}$ denotes the leading term of a nonzero $f\in z K[[z]]$.  
\begin{lem}\label{lem:basic} Let $f,g\in zK[[z]]$ be nonzero and write $f=\sum_{n\geq1}a_nz^n$ and $g=\sum_{m\geq1}b_mz^m$ for some $a_n,b_m\in K$. Then the composition $f\circ g$ given by formal substitution,  
\[f\circ g=\sum_{n\geq1}a_n\Big(\sum_{m\geq1}b_mz^m\Big)^n,\]
yields a well defined element of $zK[[z]]$. In particular, we have that \vspace{.1cm} 
\[
\ord(f\circ g)=\ord(f)\cdot \ord(g),\;\;\;\textup{lt}(f\circ g)=\textup{lt}(f)\circ\textup{lt}(g),\;\;\;\text{and}\;\;\;\textup{lc}(f\circ g)=\textup{lc}(f)\cdot\textup{lc}(g)^{\ord(f)}
 \vspace{.1cm} 
\] 
for all $f,g\in zK[[z]]$. Moreover, the composition of power series is right cancellative: if $f\circ g=h\circ g$ for some nonzero $f,g,h\in zK[[z]]$, then $f=h$.   
\end{lem}
\begin{proof}
The claim about orders and leading coefficients is a straightforward calculation using only the definition of composition. Likewise, composition of series is linear in the outer argument, so that if $f\circ g= h\circ g$ then $(f-h)\circ g=0$ in $K[[z]]$. But then if $f-h$ is non-zero, it must be the case that $\ord((f-h)\circ g)=\ord(f-h)\cdot \ord(g)$, and we obtain a contradiction.   
\end{proof}
Next, we set up the ping-pong action using valuations and affine linear transformations; see \cite[\S II.B]{de2000topics} for a classic version of the ping-pong lemma for semigroups and see \cite{kolpakov2022free} for examples coming from affine linear transformations. In what follows, $\langle \textup{lc}(f_1),\dots,\textup{lc}(f_r)\rangle$ denotes the multiplicative semigroup in $K^\times$ generated by the leading coefficients of the $f$'s. 
\begin{lem}\label{lem:ping-pong}
Let $f_1,\dots,f_r$ satisfy the conditions of Theorem
\ref{thm:free+power+series}, and let $\mathcal{S}=\langle f_1,\dots,f_r\rangle$.
Then the functions $v_1,\dots,v_r$ defined by $v_i(\textup{lc}(f_i))=1$ and
$v_i(\textup{lc}(f_j))=0$ for $i\neq j$ extend uniquely to homomorphisms
\vspace{.1cm}
\[
v_1,\dots,v_r:\langle\textup{lc}(f_1),\dots,\textup{lc}(f_r)\rangle\rightarrow\mathbb{Z}
\vspace{.1cm}
\]
from the multiplicative semigroup generated by the leading coefficients to
the additive group $\mathbb{Z}$. Moreover, $\textup{lc}(F)\in
\langle\textup{lc}(f_1),\dots,\textup{lc}(f_r)\rangle$ for all
$F\in\mathcal{S}$, so for each $1\leq i\leq r$ and $F\in\mathcal{S}$ we may
define the affine linear transformation
\vspace{.05cm}
\[
\lambda_i(F)(t):=\frac{1}{\ord(F)}\,t+\frac{v_i(\textup{lc}(F))}{\ord(F)}
\in\aff(\mathbb{R}).
\vspace{.05cm}
\]
Then the following statements hold.
\vspace{.2cm}
\begin{enumerate}
\item[\textup{(1)}] For all $1\leq i\leq r$, the map
$\lambda_i:\mathcal{S}\rightarrow\aff(\mathbb{R})$ given by
$F\mapsto\lambda_i(F)$ is an antihomomorphism:
\[\lambda_i(F\circ G)=\lambda_i(G)\circ\lambda_i(F)\]
for all $F,G\in\mathcal{S}$. Hence, the product map
$\lambda:\mathcal{S}\rightarrow\aff(\mathbb{R})^r$ given by
\[\lambda(F)=(\lambda_1(F),\dots,\lambda_r(F))\] 
is also an antihomomorphism,
where $\aff(\mathbb{R})^r$ acts on $\mathbb{R}^r$ coordinatewise.
\vspace{.3cm}
\item[\textup{(2)}] Define intervals $U_i=\big(0,\frac{2}{\ord(f_i)}\big)$,
$A_i=\big(0,\frac{1}{\ord(f_i)}\big]$, and
$B_i=\big(\frac{1}{\ord(f_i)},\frac{2}{\ord(f_i)}\big)$. Then
\[
\lambda_i(f_j)(U_i)\subseteq
\begin{cases} A_i & i\neq j, \\
              B_i & i=j.
\end{cases}
\vspace{.25cm}
\]
\item[\textup{(3)}] Define the sets $X=\prod_{i=1}^rU_i$ and
$X_i=\prod_{j=1}^rY_j$, where
\[
Y_j=\begin{cases} A_j & j\neq i, \\
                  B_j & j=i.
\end{cases}
\]
Then $X_i\subseteq X$ and $\lambda(f_i)(X)\subseteq X_i$ for all $i$, and
$X_i\cap X_j=\varnothing$ for all $i\neq j$.
\end{enumerate}
\end{lem}

\begin{proof}
Since $\textup{lc}(f_1),\dots,\textup{lc}(f_r)$ are multiplicatively
independent, every element of the multiplicative semigroup 
$\langle\textup{lc}(f_1),\dots,\textup{lc}(f_r)\rangle$ can be written
uniquely as $\textup{lc}(f_1)^{n_1}\cdots\textup{lc}(f_r)^{n_r}$ with
$(n_1,\dots,n_r)\in\mathbb{N}^r\setminus\{\mathbf{0}\}$. Hence $v_i$ is well defined by
$v_i\big(\textup{lc}(f_1)^{n_1}\cdots\textup{lc}(f_r)^{n_r}\big)=n_i$, and
this is clearly the unique extension to a homomorphism. Next, by Lemma \ref{lem:basic} we have
\begin{equation}\label{eq:recap}
\ord(F\circ G)=\ord(F)\cdot\ord(G)
\qquad\text{and}\qquad
\textup{lc}(F\circ G)=\textup{lc}(F)\cdot\textup{lc}(G)^{\ord(F)}
\end{equation}
for all $F,G\in\mathcal{S}$. Thus, induction on word length shows that $\textup{lc}(F)\in\langle\textup{lc}(f_1),\dots,\textup{lc}(f_r)\rangle$ and $\ord(F)\geq2$ for all $F\in\mathcal{S}$, so each $\lambda_i(F)$ is well defined.

For statement (1), fix $i$ and write $d_F=\ord(F)$ and
$w_F=v_i(\textup{lc}(F))$ for $F\in\mathcal{S}$. Since $v_i$ is a
homomorphism, \eqref{eq:recap} gives $d_{F\circ G}=d_Fd_G$ and
$w_{F\circ G}=w_F+d_Fw_G$. Hence
\[
\lambda_i(F\circ G)(t)=\frac{t+w_F+d_Fw_G}{d_Fd_G}
=\frac{1}{d_G}\Big(\frac{t+w_F}{d_F}\Big)+\frac{w_G}{d_G}
=\lambda_i(G)\big(\lambda_i(F)(t)\big).
\]
Since composition in $\aff(\mathbb{R})^r$ is coordinatewise, $\lambda$ is also an antihomomorphism.

For statement (2), write $d_j=\ord(f_j)$, and note that $d_j\geq2$ for all $j$. By the definition of $v_1,\dots,v_r$, we have
\[
\lambda_i(f_i)(t)=\frac{t+1}{d_i}
\qquad\text{and}\qquad
\lambda_i(f_j)(t)=\frac{t}{d_j}\quad\text{for $j\neq i$}.
\]
Now let $t\in U_i$, so that $0<t<2/d_i$. Since $\lambda_i(f_i)$ is
strictly increasing and $d_i\geq2$,
\[
\frac{1}{d_i}=\lambda_i(f_i)(0)<\lambda_i(f_i)(t)
<\lambda_i(f_i)\Big(\frac{2}{d_i}\Big)=\frac{2}{d_i^2}+\frac{1}{d_i}
\leq\frac{2}{d_i}.
\]
Therefore, $\lambda_i(f_i)(U_i)\subseteq B_i$. Likewise, for $j\neq i$, since $\lambda_i(f_j)$ is strictly increasing and $d_j\geq2$,
\[0=\lambda_i(f_j)(0)<\lambda_i(f_j)(t)<\lambda_i(f_j)\Big(\frac{2}{d_i}\Big)=\frac{2}{d_id_j}\leq\frac{1}{d_i}.\]
Hence, $\lambda_i(f_j)(U_i)\subseteq A_i$ as claimed.

For statement (3), note that $A_j,B_j\subseteq U_j$ for all $j$, so
$X_i\subseteq X$. If $\mathbf{t}=(t_1,\dots,t_r)\in X$, then the $j$-th coordinate of $\lambda(f_i)(\mathbf{t})$ is $\lambda_j(f_i)(t_j)$, which lies in $B_j$ if $j=i$ and in $A_j$ if $j\neq i$ by statement (2). Hence $\lambda(f_i)(X)\subseteq X_i$. Finally, if $i\neq j$, then the $i$-th factor of $X_i$ is $B_i$ while the $i$-th factor of $X_j$ is $A_i$; since $A_i\cap B_i=\varnothing$, we conclude that $X_i\cap X_j=\varnothing$.
\end{proof}
\begin{proof}[Proof of Theorem \ref{thm:free+power+series}]
We show by induction on $m\geq1$ that if
$\theta_1\circ\dots\circ\theta_n=\tau_1\circ\dots\circ\tau_m$ with $n\geq m$
and $\theta_k,\tau_l\in\{f_1,\dots,f_r\}$, then $n=m$ and $\theta_k=\tau_k$
for all $k$. Applying $\lambda$ and using Lemma \ref{lem:ping-pong} part (1) gives
\[
\lambda(\theta_n)\circ\dots\circ\lambda(\theta_1)
=\lambda(\tau_m)\circ\dots\circ\lambda(\tau_1).
\]
Write $\theta_n=f_i$ and $\tau_m=f_j$, and fix $\mathbf{p}\in X$. Since $\lambda(f_k)(X)\subseteq X_k\subseteq X$ for all $k$ by Lemma
\ref{lem:ping-pong} part (3), evaluating both sides at $\mathbf{p}$ produces a point of $X_i\cap X_j$. Hence $i=j$, that is, $\theta_n=\tau_m$. If $m=1$ and $n>1$, then right cancellation (Lemma \ref{lem:basic}) applied to $(\theta_1\circ\dots\circ\theta_{n-1})\circ\theta_n=z\circ\theta_n$ gives $\theta_1\circ\dots\circ\theta_{n-1}=z$, which is impossible since the left side has order at least $2$; so $n=1$. If $m>1$, then right cancellation implies that $\theta_1\circ\dots\circ\theta_{n-1}=\tau_1\circ\dots\circ\tau_{m-1}$, from which the result follows by induction.
\end{proof}
In particular, we may apply the result above to semigroups of rational functions with a common super-attracting fixed point. Recall that a fixed point $P\in\mathbb{C}$ of $f\in\mathbb{C}(z)$ is
\emph{super-attracting} if $f'(P)=0$. If $f(\infty)=\infty$, we say that $\infty$ is super-attracting if $0$ is a super-attracting fixed point of $1/f(1/z)$. Since the multiplier of a fixed point is invariant under conjugation by $\textup{PGL}_2(\mathbb{C})$ (see \cite[\S1.3]{SilvDyn}), a fixed point $P\in\mathbb{P}^1(\mathbb{C})$ of $f$ is super-attracting if and only if $0$ is a super-attracting fixed point of $L^{-1}\circ f\circ L$ for some (equivalently, every) $L\in\textup{PGL}_2(\mathbb{C})$ with
$L(0)=P$. Before we state our result on semigroups of rational functions, we need the following technical facts about super-attracting fixed points.

\begin{lem}\label{lem:taylor}
Let $P\in\mathbb{P}^1(\mathbb{C})$, and let $L\in\textup{PGL}_2(\mathbb{C})$ be such that $L(0)=P$. Suppose that $f,g\in\mathbb{C}(z)$ are non-constant and that $P$ is a super-attracting fixed point of both $f$ and $g$. Set $F:=L^{-1}\circ f\circ L$ and $G:=L^{-1}\circ g\circ L$. Then the following statements hold:
\vspace{.1cm}
\begin{enumerate}
\item[\textup{(1)}] The maps $F$, $G$, and $F\circ G$ have super-attracting fixed points at $z=0$. In particular, $P$ is a super-attracting fixed point of $f\circ g$. Moreover, the Taylor series
$T_F,T_G,T_{F\circ G}$ of $F$, $G$, and $F\circ G$ about $z=0$ are nonzero elements of $z^2\mathbb{C}[[z]]$ and
$2\leq\textup{ord}(T_F)\leq\deg(f)$.
\vspace{-.25cm}
\item[\textup{(2)}] We have that $T_{F\circ G}=T_F\circ T_G$,
where the composition on the right is given by formal substitution as in Lemma \ref{lem:basic}.
\vspace{.2cm}
\item[\textup{(3)}] We have that $F=G$ if and only if $T_F=T_G$.\vspace{.1cm} 
\end{enumerate}
In particular, the set
\[\mathcal{S}_P=\{f\in\mathbb{C}(z)\,:\,\deg(f)>0\ \text{and $P$ is a super-attracting fixed point of $f$}\}\]
is a semigroup under composition, and the map $T_{P,L}:\mathcal{S}_P\rightarrow z^2\mathbb{C}[[z]]$ given by $f\mapsto T_{L^{-1}\circ f\circ L}$ is an injective semigroup homomorphism.
\end{lem}

\begin{proof}
Throughout, we say that $h\in\mathbb{C}(z)$ is \emph{regular at $0$} if $h=a/b$ for some $a,b\in\mathbb{C}[z]$ with $b(0)\neq0$. Such functions form a subring of $\mathbb{C}(z)$, and the map $h\mapsto T_h$ is a ring homomorphism from this subring to $\mathbb{C}[[z]]$. Note that polynomials are their own Taylor series, and that if $h=a/b$ as above, then $T_h=a\,b^{-1}$, where $b^{-1}$ denotes the inverse of $b$ in $\mathbb{C}[[z]]$, which exists since $b(0)\neq0$.

For statement (1), we note that by the remarks above, $F(0)=G(0)=0$ and $F'(0)=G'(0)=0$. Hence $(F\circ G)(0)=0$ and $(F\circ G)'(0)=F'(0)G'(0)=0$. Since $F\circ G=L^{-1}\circ(f\circ g)\circ L$, the point $P$ is a super-attracting fixed point of $f\circ g$. Now let $H\in\{F,G,F\circ G\}$. Since $H(0)=0$, we may write $H=p/q$ in lowest terms with $p,q\in\mathbb{C}[z]$ and $q(0)\neq0$. In particular, $H$ is regular at $0$, and the constant and linear coefficients of $T_H$ are $H(0)=0$ and $H'(0)=0$, so $T_H\in z^2\mathbb{C}[[z]]$. Moreover, $H$ is non-constant, so $p\neq0$,
and hence $T_H=p\,q^{-1}\neq0$. Finally, since $q^{-1}$ is a unit in
$\mathbb{C}[[z]]$ and conjugation preserves degree, we have
\[
2\leq\textup{ord}(T_F)=\textup{ord}(p)\leq\deg(p)\leq\deg(F)=\deg(f)
\]
when $H=F$.

For statement (2), since $T_G\in z\mathbb{C}[[z]]$, the substitution map $\sigma(\phi):=\phi\circ T_G$ is a ring homomorphism $\mathbb{C}[[z]]\to\mathbb{C}[[z]]$; note that $\sigma(a)=a\circ T_G$ for $a\in\mathbb{C}[z]$. Write $F=p/q$ with $p,q\in\mathbb{C}[z]$ and $q(0)\neq0$, so that $T_F=p\,q^{-1}$ in $\mathbb{C}[[z]]$. Now $F\circ G=(p\circ G)(q\circ G)^{-1}$, where $p\circ G$ and $q\circ G$ are regular at $0$ and $(q\circ G)(0)=q(0)\neq0$. Therefore
\[
T_{F\circ G}=T_{p\circ G}\,T_{q\circ G}^{-1}
=\sigma(p)\,\sigma(q)^{-1}=\sigma(p\,q^{-1})=\sigma(T_F)=T_F\circ T_G,
\]
where the first and second equalities use that $h\mapsto T_h$ is a ring homomorphism, and the third uses that $\sigma$ is.

For statement (3), if $F=G$, then clearly $T_F=T_G$. Conversely, suppose that $T_F=T_G$. Since $F$ and $G$ are regular at $0$, so is $F-G$; write $F-G=r/s$ with $r,s\in\mathbb{C}[z]$ and $s(0)\neq0$. Then
\[
r\,s^{-1}=T_{F-G}=T_F-T_G=0
\]
in $\mathbb{C}[[z]]$. Multiplying by $s$ gives $r=0$, and hence $F=G$.

Finally, if $f,g\in\mathcal{S}_P$, then $\deg(f\circ g)=\deg(f)\deg(g)>0$,
and $P$ is a super-attracting fixed point of $f\circ g$ by (1); hence $\mathcal{S}_P$ is closed under composition. Moreover, $z^2\mathbb{C}[[z]]$ is a semigroup under composition by Lemma \ref{lem:basic}, and $T_{P,L}(f)\in z^2\mathbb{C}[[z]]$ by (1). Since $L^{-1}\circ(f\circ g)\circ L=F\circ G$, statement (2) gives
$T_{P,L}(f\circ g)=T_{P,L}(f)\circ T_{P,L}(g)$. Injectivity follows from (3), since $f=L\circ F\circ L^{-1}$.
\end{proof}

In what follows, if $f\in\mathbb{C}(z)$ is non-constant with super-attracting fixed point $P$ and $L\in\textup{PGL}_2(\mathbb{C})$ is such that $L(0)=P$, then we call the leading coefficient of the Taylor series $T_{L^{-1}\circ f\circ L}$ centered about $z=0$,  
\[\mathfrak{b}_{P,L}(f)=\textup{lc}\big(T_{L^{-1}\circ f\circ L}\big),\]
the \emph{B\"{o}ttcher coefficient} of $f$ with respect to $P$ and $L$. In particular, with the results above in place, we have the following criterion for freeness of semigroups of maps with a common super-attracting fixed point.
\begin{cor}\label{cor:Bottcher}
Let $f_1,\dots,f_r\in \mathbb{C}(z)$ be rational functions all of degree at least $2$ and all with a common super-attracting fixed point $P\in\mathbb{P}^1(\mathbb{C})$. Choose $L\in\textup{PGL}_2(\mathbb{C})$ such that $L(0)=P$. If the multiplicative semigroup $\langle \mathfrak{b}_{P,L}(f_1),\dots, \mathfrak{b}_{P,L}(f_r)\rangle$ in $\mathbb{C}^\times$ is free commutative of rank $r$, then $\langle f_1,\dots,f_r\rangle$ is free of rank $r$.       
\end{cor}
\begin{proof} By Lemma \ref{lem:taylor}, the map $\langle f_1,\dots, f_r \rangle\rightarrow \langle T_{L^{-1}\circ f_1\circ L},\dots, T_{L^{-1}\circ f_r\circ L}\rangle $ given by $f\rightarrow T_{L^{-1}\circ f\circ L}$ is a semigroup isomorphism. On the other hand, if $\langle \mathfrak{b}_{P,L}(f_1),\dots, \mathfrak{b}_{P,L}(f_r)\rangle$ in $\mathbb{C}^\times$ is free commutative of rank $r$, then the given generators must be a free basis (hence, multiplicatively independent). Therefore, Theorem \ref{thm:free+power+series} implies that $\langle T_{L^{-1}\circ f_1\circ L},\dots, T_{L^{-1}\circ f_r\circ L}\rangle$ is free with the given basis, from which the desired claim follows.         
\end{proof}
In particular, we obtain the following improvement of \cite[Theorem 5.1]{MR4449715}. Specifically, we remove the condition that the degrees of the polynomials also be multiplicatively independent.  
\begin{cor}\label{cor:poly} Let $K$ be a field of characteristic zero and let $f_1,\dots,f_r\in K[x]$ be polynomials of degree at least $2$. If the leading coefficients of the $f$'s are multiplicatively independent, then $\langle f_1,\dots,f_r\rangle$ is a free semigroup of rank $r$.
\end{cor}
\begin{proof} Let $K_0\subseteq K$ be the subfield generated over $\mathbb{Q}$ by the coefficients of $f_1,\dots,f_r$. Since $K_0$ is finitely generated over $\mathbb{Q}$, there is a field embedding $\iota:K_0\hookrightarrow\mathbb{C}$: write $K_0$ as a finite extension of a purely transcendental extension $\mathbb{Q}(t_1,\dots,t_n)$, send $t_1,\dots,t_n$ to algebraically independent complex numbers, and extend to $K_0$ using that $\mathbb{C}$ is
algebraically closed; see, e.g., \cite[Chapter~V, Theorem~2.8]{Lang}. Applying $\iota$ to coefficients gives an injective map
$K_0[z]\to\mathbb{C}[z]$, $f\mapsto f^\iota$, that preserves degrees and composition, and hence induces a semigroup isomorphism
$\langle f_1,\dots,f_r\rangle\cong\langle f_1^\iota,\dots,f_r^\iota\rangle$. Moreover, since $\iota$ restricts to an injective homomorphism $K_0^\times\to\mathbb{C}^\times$ and sends the leading coefficient of $f_i$ to that of $f_i^\iota$, the leading coefficients of $f_1^\iota,\dots,f_r^\iota$ are multiplicatively independent. Thus, we may assume that $K=\mathbb{C}$.

Polynomials of degree at least $2$ have a common super-attracting fixed point at $P=\infty$. In particular, $L(z)=1/z$ fits the description of Lemma \ref{lem:taylor}. On the other hand, it is straightforward to check that if $f=a_d z^d+\dots +a_0$ for some $a_i\in\mathbb{C}$ with $a_d\neq0$, then the Taylor series of $L^{-1}\circ f\circ L$ satisfies 
\[T_{L^{-1}\circ f\circ L}=\frac{1}{a_d}z^d+O(z^{d+1}).\]
Hence, $\mathfrak{b}_{\infty,L}(f_i)$ is the inverse of the usual leading coefficient of the polynomial $f_i$. Moreover, the set of complex numbers $\mathfrak{b}_{\infty,L}(f_1), \dots, \mathfrak{b}_{\infty,L}(f_r)$ is multiplicatively independent if and only if $\mathfrak{b}_{\infty,L}(f_1)^{-1}, \dots, \mathfrak{b}_{\infty,L}(f_r)^{-1}$ is multiplicatively independent. The claim follows.    
\end{proof} 
Finally, we note that the proportion of tuples of polynomials over a fixed number field $K$ for which the corresponding semigroup is not free is quite small (when we order coefficients by height and view the degrees as fixed). To make this statement more precise, we define the height $H(f)$ of a polynomial $f(z)=a_dz^d+\dots+a_0\in K[z]$ to be 
\[H(f):=\max\{H(a_0),\dots, H(a_d)\},\] 
i.e., the max height of $f$'s coefficients; see \cite[\S3.1]{SilvDyn} for the definition of the absolute Weil height $H(a)$ for $a\in K$. Moreover, for a given sequence of degrees $d_1,\dots, d_r$, we let $\mathcal{P}_{d_1,\dots,d_r}(K,x)$ denote the set of $r$-tuples of polynomials over $K$ with the given degrees whose components all have height at most $x$:
\[\mathcal{P}_{d_1,\dots,d_r}(K,x):=\Big\{(f_1,\dots,f_r)\,:\, f_i\in K[z],\; \deg(f_i)=d_i,\;\text{and}\; H(f_i)\leq x\;\; \text{for all $1\leq i\leq r$}\Big\}. 
\vspace{.1cm} 
\]
Likewise, we let $\mathcal{P}_{d_1,\dots,d_r}(K,x)^+$ be the set of tuples of polynomials with the same defining characteristics as $\mathcal{P}_{d_1,\dots,d_r}(K,x)$ except that we allow $\deg(f_i)\leq d_i$, instead of forcing an equality; that is, we allow zeros in certain slots. Finally, we let
\[\mathcal{N}_{d_1,\dots,d_r}(K,x):=\{(f_1,\dots, f_r)\in \mathcal{P}_{d_1,\dots,d_r}(K,x)\,:\, \langle f_1,\dots,f_r\rangle\;\text{is not  free of rank $r$}\}
\vspace{.1cm}
\]
denote the collections of such polynomials that do not generate a free semigroup (on the given basis) under composition. Then the following statement is a consequence of Corollary \ref{cor:poly} and \cite[Theorem 1.2]{PappalardiShaShparlinskiStewart2018}.   
\begin{cor}{\label{cor:non-free-proportion}} Let $K$ be a fixed number field and let $r,d_1,\dots,d_r\geq2$. Then 
\vspace{.1cm}
\[\frac{\# \mathcal{N}_{d_1,\dots,d_r}(K,x)}{\#\mathcal{P}_{d_1,\dots,d_r}(K,x)}=O\big(x^{-2[K:\mathbb{Q}]}\big). 
\vspace{.1cm}\]
In particular, 
\[\displaystyle{\lim_{x\rightarrow\infty}\ \frac{\#\mathcal{N}_{d_1,\dots,d_r}(K,x)}{\#\mathcal{P}_{d_1,\dots,d_r}(K,x)}=0}.
\vspace{.1cm} 
\] 
Thus, when ordered by height, most collections of polynomials over $K$ of a given degree sequence generate free semigroups (on the given basis) under composition.  
\end{cor}
\begin{proof} We begin with some notation. Write $D=[K:\mathbb{Q}]$ and $\mathcal{P}^+=\mathcal{P}_{d_1-1,\dots,d_r-1}(K,x)^+$, and consider the map $\pi:\mathcal{P}_{d_1,\dots,d_r}(K,x)\to(K^\times)^r$ given by $\pi(f_1,\dots,f_r)=(\textup{lc}(f_1),\dots,\textup{lc}(f_r))$. By Corollary \ref{cor:poly}, $\pi(\mathcal{N}_{d_1,\dots,d_r}(K,x))$ is contained in the set $\mathcal{L}_{r,K}(x)$ of multiplicatively dependent vectors in
$(K^\times)^r$ whose coordinates have height at most $x$; in the notation of \cite{PappalardiShaShparlinskiStewart2018}, $\#\mathcal{L}_{r,K}(x)=L^*_{r,K}(x)$. Moreover, writing each polynomial as its leading term plus a polynomial of lower degree identifies every fiber of $\pi$ with $\mathcal{P}^+$. Hence, by
\cite[Theorem 1.2]{PappalardiShaShparlinskiStewart2018},
\[
\#\mathcal{N}_{d_1,\dots,d_r}(K,x)
\leq\sum_{v\in\mathcal{L}_{r,K}(x)}\#\pi^{-1}(v)
=L^*_{r,K}(x)\cdot\#\mathcal{P}^+
=O\big(x^{2D(r-1)}\big)\cdot\#\mathcal{P}^+ .
\]
On the other hand, the same identification gives
\[
\#\mathcal{P}_{d_1,\dots,d_r}(K,x)=B_K(x)^r\cdot\#\mathcal{P}^+,
\qquad\text{where } B_K(x):=\#\{a\in K^\times: H(a)\leq x\},
\]
and Schanuel's theorem \cite{schanuel1979heights} (see also
\cite[(1.5)]{PappalardiShaShparlinskiStewart2018} for the formulation with the absolute height) gives $B_K(x)=C(K)x^{2D}+o(x^{2D})$ for some constant
$C(K)>0$. Dividing the two estimates above proves the claim.
\end{proof}
\begin{rmk} In fact, it may be possible to make the corresponding lower bound on the the proportion of free semigroups larger by allowing a change of variables first (e.g., a substitution of the form $z=cy$ for some $c\in K^\times$) when proving freeness in Corollary \ref{cor:poly}. In particular, it is possible that this type of argument could lead to a power saving estimate on  $\mathcal{N}_{d_1,\dots,d_r}(K,x)$ in Corollary \ref{cor:non-free-proportion}.   \end{rmk}

\section*{A second proof of Theorem \ref{thm:free+power+series}}
The following argument is due to Fedor Pakovich. Let $f_1,\dots,f_r$ satisfy the conditions of Theorem \ref{thm:free+power+series}, and write $a_s=\textup{lc}(f_s)$ and $d_s=\ord(f_s)$ for $1\leq s\leq r$; in particular, $d_s\geq2$ for all $s$. Consider a word
\[
 W=f_{i_1}\circ\cdots\circ f_{i_n}
\]
and set
\begin{equation}\label{eq:D}
 D(W)=d_{i_1}d_{i_2}\cdots d_{i_n}.
\end{equation}
Then $\ord(W)=D(W)$, and repeated application of Lemma \ref{lem:basic} gives
\begin{equation}\label{eq:E-prefix}
 \textup{lc}(W)=\prod_{s=1}^r a_s^{E_s(W)},
\end{equation}
where, for $s\in\{1,\ldots,r\}$,
\begin{equation}\label{eq:E}
 E_s(W)=\sum_{\substack{1\le k\le n\\i_k=s}}
       \prod_{\ell<k}d_{i_\ell}
       =\sum_{\substack{1\le k\le n\\i_k=s}}
       \frac{D(W)}{d_{i_k}d_{i_{k+1}}\cdots d_{i_n}}.
\end{equation}
Note that all the fractions in~\eqref{eq:E} are integers, independently
of the field $K$.

Multiplicative independence of the coefficients implies that
\begin{equation}\label{eq:exponent-equality}
 \textup{lc}(W)=\textup{lc}(V)\quad\Longrightarrow\quad
 E_s(W)=E_s(V)\ \text{for all }s.
\end{equation}

We use the coefficient estimate from the proof of Theorem~2.3
in~\cite[pp.~13835--13836]{pakovich2022sharing}.
Set $D=D(W)$. For each fixed $s$, we have
\begin{align}
 W\text{ ends in }f_s
 &\quad\Longrightarrow\quad E_s(W)\ge\frac{D}{d_s},
 \label{eq:last-s}\\
 W\text{ ends in }f_j,\quad j\ne s
 &\quad\Longrightarrow\quad E_s(W)<\frac{D}{d_s}.
 \label{eq:last-other}
\end{align}
Indeed, if $i_n=s$, then the term $k=n$ in~\eqref{eq:E} equals $D/d_s$, which gives~\eqref{eq:last-s}. On the other hand, if $i_n\neq s$, then every index $k$ with $i_k=s$ satisfies $k<n$, and the corresponding term in~\eqref{eq:E} is
\[
\frac{D}{d_s}\prod_{\ell=k+1}^{n}\frac{1}{d_{i_\ell}}\leq\frac{D}{d_s}\cdot\frac{1}{2^{\,n-k}},
\]
since $d_{i_\ell}\geq2$ for all $\ell$. Summing over distinct $k<n$ gives $E_s(W)\leq\frac{D}{d_s}\big(1-2^{-(n-1)}\big)<\frac{D}{d_s}$, which proves~\eqref{eq:last-other}.

Now suppose that two nonempty words $W,V$ represent the same formal power series.
Their orders are equal to some $D$, and their leading coefficients coincide.
Hence all the exponents $E_s$ coincide as well.

If $W$ ends in $f_s$ and $V$ ends in $f_j$ with $j\ne s$, then
\[
 E_s(W)\ge D/d_s>E_s(V),
\]
contrary to the equality of the exponents.
Thus both words end in the same letter $f_s$:
\[
 W=W'\circ f_s,\qquad V=V'\circ f_s.
\]
Cancelling $f_s$ on the right gives $W'=V'$ by Lemma \ref{lem:basic}.
We continue cancelling the common final letters. Both words must be exhausted
simultaneously: otherwise, the identity series of order $1$ would equal
a nonempty word of order at least~$2$. \qed

\medskip
\textbf{Acknowledgments:} We thank the American Institute of Mathematics for hosting the workshop, Dynamics of Multiple Maps, in November 2025, which inspired some of the ideas in this paper. We also thank Fedor Pakovich for helpful correspondence and for allowing us to include his proof of Theorem \ref{thm:free+power+series}.
\bibliographystyle{plain}
\bibliography{main}

@article {MR4780496,
    AUTHOR = {Bell, Jason P. and Huang, Keping and Peng, Wayne and Tucker,
              Thomas J.},
     TITLE = {A {T}its alternative for endomorphisms of the projective line},
   JOURNAL = {J. Eur. Math. Soc. (JEMS)},
  FJOURNAL = {Journal of the European Mathematical Society (JEMS)},
    VOLUME = {26},
      YEAR = {2024},
    NUMBER = {12},
     PAGES = {4903--4922},
      ISSN = {1435-9855,1435-9863},
   MRCLASS = {20M05 (11G50 14H37)},
  MRNUMBER = {4780496},
MRREVIEWER = {Fei\ Hu},
       DOI = {10.4171/jems/1376},
       URL = {https://doi.org/10.4171/jems/1376},
}

@article {MR4449715,
    AUTHOR = {Hindes, Wade},
     TITLE = {Counting points of bounded height in monoid orbits},
   JOURNAL = {Math. Z.},
  FJOURNAL = {Mathematische Zeitschrift},
    VOLUME = {301},
      YEAR = {2022},
    NUMBER = {4},
     PAGES = {3395--3416},
      ISSN = {0025-5874,1432-1823},
   MRCLASS = {37P15 (11D45 11G50 37P05)},
  MRNUMBER = {4449715},
       DOI = {10.1007/s00209-022-03021-8},
       URL = {https://doi.org/10.1007/s00209-022-03021-8},
}

@article{PappalardiShaShparlinskiStewart2018,
  author    = {Pappalardi, Francesco and Sha, Min and Shparlinski, Igor E. and Stewart, Cameron L.},
  title     = {On multiplicatively dependent vectors of algebraic numbers},
  journal   = {Transactions of the American Mathematical Society},
  volume    = {370},
  number    = {9},
  pages     = {6173--6196},
  year      = {2018},
  doi       = {10.1090/tran/7178},
  url       = {https://doi.org}
}

@book{de2000topics,
  title={Topics in geometric group theory},
  author={de La Harpe, Pierre},
  year={2000},
  publisher={University of Chicago Press}
}

@article{Zieve,
  title={Functional equations in polynomials},
  author={Jiang, Zhan and Zieve, Michael E},
  journal={arXiv preprint arXiv:2008.09554},
  year={2020}
}

@article{schanuel1979heights,
  title={Heights in number fields},
  author={Schanuel, Stephen Hoel},
  journal={Bulletin de la Soci{\'e}t{\'e} math{\'e}matique de France},
  volume={107},
  pages={433--449},
  year={1979}
}

@article{beaumont2025uniformtitsalternativeendomorphisms,
      title={A uniform Tits alternative for endomorphisms of the projective line}, 
      author={Alonso Beaumont},
      journal={Algebra \& Number Theory},
      note={To appear},
      year={2025},
      eprint={2504.14263},
      archivePrefix={arXiv},
      primaryClass={math.NT},
      url={https://arxiv.org},
}

@book {SilvDyn,
    AUTHOR = {Silverman, Joseph H.},
     TITLE = {The arithmetic of dynamical systems},
    SERIES = {Graduate Texts in Mathematics},
    VOLUME = {241},
 PUBLISHER = {Springer, New York},
      YEAR = {2007},
     PAGES = {x+511},
      ISBN = {978-0-387-69903-5},
   MRCLASS = {11-02 (11-01 11G05 11G07 11G50 37-02 37F10)},
  MRNUMBER = {2316407},
MRREVIEWER = {Thomas Ward},
       DOI = {10.1007/978-0-387-69904-2},
       URL = {https://doi.org/10.1007/978-0-387-69904-2},
}

@article{kolpakov2022free,
  title={On free semigroups of affine maps on the real line},
  author={Kolpakov, Alexander and Talambutsa, Alexey},
  journal={Proceedings of the American Mathematical Society},
  volume={150},
  number={6},
  pages={2301--2307},
  year={2022}
}

@article{pakovich2026right,
  title={Right amenability in semigroups of formal power series},
  author={Pakovich, Fedor},
  journal={Groups, Geometry, and Dynamics},
  volume={20},
  number={1},
  pages={205},
  year={2026},
  publisher={European Mathematical Society (EMS)}
}

@book{Lang,
  author    = {Lang, Serge},
  title     = {Algebra},
  edition   = {third},
  series    = {Graduate Texts in Mathematics},
  volume    = {211},
  publisher = {Springer-Verlag, New York},
  year      = {2002},
  pages     = {xvi+914},
  isbn      = {0-387-95385-X},
  doi       = {10.1007/978-1-4613-0041-0},
  mrnumber  = {1878556},
  note      = {Revised third edition}
}

@article{pakovich2022sharing,
  author        = {Pakovich, Fedor},
  title         = {Sharing a measure of maximal entropy in polynomial semigroups},
  journal       = {Int. Math. Res. Not. IMRN},
  year          = {2022},
  volume        = {2022},
  number        = {18},
  pages         = {13829--13840},
  eprint        = {2009.12261},
  archivePrefix = {arXiv}
}

\end{document}